\documentclass[10pt,letterpaper,twocolumn]{article}

\usepackage{authblk}
\usepackage{abstract}
\usepackage{titlesec}
\usepackage{geometry}
\usepackage{multicol, multirow}
\usepackage{graphicx}
\usepackage[labelsep=period,font=it]{caption}
\usepackage{booktabs}
\usepackage[hang, flushmargin]{footmisc}
\usepackage[fleqn]{amsmath}
\usepackage{mathtools}
\usepackage{natbib}
\usepackage[utf8]{inputenc}
\usepackage[english]{babel}
\usepackage{amsmath, amssymb, amsfonts, mathptmx, setspace, enumitem}
\usepackage{amsthm}
\setitemize{noitemsep,topsep=0pt,parsep=0.1ex,partopsep=0pt,leftmargin=*}
\setenumerate{topsep=0pt,parsep=0pt,partopsep=0pt,leftmargin=*}

\newcommand{\G}{{\bf G}} 
\newcommand{\e}{{\bf e}} 
\newcommand{\X}{{\bf X}} 
\theoremstyle{plain}
\newtheorem{theorem}{Theorem}
\newtheorem{lemma}[theorem]{Lemma}
\newtheorem{prop}[theorem]{Proposition}

\graphicspath{{./figures/}}

\titleformat*{\section}{\normalsize\bf}
\titlespacing{\section}{0pt}{18pt}{6pt}
\titleformat*{\subsection}{\normalsize\it}
\titlespacing{\subsection}{0pt}{12pt}{6pt}

\providecommand{\keywords}[1]{\hspace{0.25in}\textit{Keywords}\\
\vspace{-20pt}
\begin{center}
\begin{minipage}{5.84in}
#1
\end{minipage}
\end{center}}

\title{\fontsize{14}{16.8}
\textbf{DETECTING SYMMETRY IN DESIGNING HEAT EXCHANGER NETWORKS\footnote{In Proceedings of the International Conference of Foundations of Computer-Aided Process Operations - FOCAPO/CPC 2017, Tuscon, AZ, Jan 2017. 
Edited by C. Maravelias, E. Ydstie, L. Megan and B. W. Bequette.
}}
}
\author{Georgia Kouyialis}
\author{Ruth Misener\thanks{\textit{r.misener@imperial.ac.uk; Tel: +44 (0) 20759 48315}}}
\affil{Department of Computing, Imperial College London, South Kensington Campus, SW7 2AZ, United Kingdom}
\date{}
\begin{document}

\setlength{\abovedisplayskip}{3pt}
\setlength{\belowdisplayskip}{3pt}

\pagenumbering{gobble}
\twocolumn[
\begin{@twocolumnfalse}
\maketitle
\begin{abstract}
\vspace*{-10pt}
\noindent
Symmetry in mathematical optimisation is of broad and current interest. In problem classes such as mixed-integer linear programming (MILP), equivalent solutions created by symmetric variables and constraints may combinatorially increase the search space. Identifying problem symmetries is an important step towards expediting tree-based algorithms such as branch-and-cut because computationally classifying equivalence allows state-of-the-art solver software to omit symmetric solutions. But symmetry has not been characterised in several critically important process systems engineering applications such as heat exchanger network synthesis; neither do current MILP solvers detect or use symmetries for these energy efficiency problems. This paper uses group theory to study the MILP transshipment model of heat exchanger network synthesis and identifies several types of symmetry arising in the problem. 
We also use parameters in the optimisation problem, e.g.\ temperature and heat capacities of each stream, to classify special cases with many equivalent optimal solutions.
Computational results from an online test case corroborate the proofs. 
\end{abstract}
\vspace{12pt}
\end{@twocolumnfalse}
]
\saythanks
\linespread{1}
\section*{Introduction}

\noindent
Heat recovery is a major component of industrial processes: a quarter of the 2012 European Union energy consumption came from industry and industry uses 73\% of this energy on heating and cooling \citep{ec-heating-and-cooling:2016}. Heat exchangers reuse excess process heat to save cost and improve energy efficiency by reducing utility usage. In their review articles, \cite{furman:2002} and \cite{escobar-et-trierweiler:2013} report two main synthesis approaches: pinch- and optimisation-based methodologies. Optimisation methods automatically generate the best design taking into consideration both the investment and the operation cost \citep{grossmann:1990}. But there are difficulties when we try to model and solve these problems. One source of problem complexity is the combinatorial explosion in the possible number of stream matches enhancing energy recovery \citep{floudas}. Each possible match between two streams introduces a binary decision variable, so the number of binary variables may grow quadratically with the number of streams. Solving the HENS \emph{simultaneously}, i.e.\ generating the optimal network without decomposition \citep{furman:2002}, requires a mixed-integer nonlinear programming (MINLP) formulation to account for stream mixing and the nonlinear nature of heat exchange \citep{yee:1990, Ciric:1991, papalexandri-pistikopoulos:1994}. These nonlinear terms need strong relaxation methods to approach a global optimum \citep{mistry:2016}. 

An alternative approach is the sequential formulation where the problem is decomposed into three tasks: (i) minimum utility cost, (ii) minimum number of matches, and (iii) minimum investment cost. This method optimises the following mathematical models in series: (i) the linear programming (LP) transshipment model \citep{papoulias:1983}, (ii) the mixed-integer linear programming (MILP) transshipment model \citep{papoulias:1983}, and (iii) the nonlinear programming (NLP) model based on a network superstructure solving \citep{floudas:1986}. The sequential method is less computationally difficult than the simultaneous method, but the sequential method cannot guarantee global optimality of the original problem. We choose to study symmetries in the sequential method because each subproblem nicely isolates computational difficulties associated with solving the full simultaneous model; these studies will give us a new handle on approaching simultaneous synthesis.

Figure \ref{figure:Hmatch} is based on the \cite{floudas} transshipment model. Analogously to transferring a product from source to destination via intermediate intervals, the transshipment model transfers the heat from hot streams and utilities to the cold streams and utilities via temperature intervals. The temperature change is caused by matching the hot and cold streams and utilities at each interval, so, for
two sets of 3 hot streams and 3 cold streams, there are $3 \cdot 3 = 9$ binary variables and in the worst case the MILP needs $2^9 = 512$ nodes.

\noindent
\begin{minipage}{\linewidth}
\centering
\includegraphics[width=.95 \linewidth]{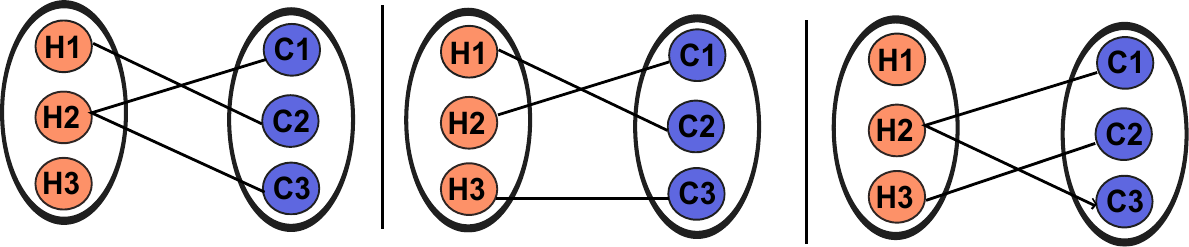}
\captionof{figure}[]{Possible configurations of hot/cold stream pairs.}\label{figure:Hmatch}
\end{minipage}

When we only consider hot-to-cold matches in each transshipment model interval, $n$ hot streams and $m$ cold streams may generate $n \cdot m$ possible pairs. There may be many MILP transshipment model solutions with equivalent objective value \citep{gundersen1990}; we posit that this is due to symmetry.
The significant role of symmetry in optimisation models is described and defined by \cite{margot:2010} and \cite{liberti:2012}. When solving the problem using a tree search strategy such as branch-and-cut, symmetric solutions are unnecessary duplications that need not be investigated. Symmetries cause exponentially large trees with long times to termination, so exploiting symmetry, e.g.\ via advanced branching strategies, may offer an
important advantage for branch-and-cut \citep{ostrowski:2009, costa-etal:2013}. 
This manuscript presents the MILP transshipment model of \textsl{HENS}, analyses the problem structure, and proves the existence of symmetry. Finally, the proofs are computationally demonstrated in a test case.  

\vspace{-12pt}
\section*{Heat Exchanger Network Synthesis}

The description, configuration and analysis of HENS follow \cite{floudas}. This paper addresses MILP transshipment models consisting of a set of hot process streams $HP$ to be cooled and a set of cold process streams $CP$ to be heated; each stream has an initial and target temperature and a heat capacity. There are also hot utilities $HU$ and cold utilities $CU$ with associated temperatures. The symbols representing the mathematical formulation are shown in Table \ref{table:hen}. 

State-of-the-art solver software CPLEX does not seem to use problem symmetry in the MILP transshipment model; turning on-and-off the symmetry feature does not to alter the search tree of CPLEX 12.6. We posit that detecting symmetry is particularly difficult because 
the heat that is provided or required by the streams is a continuous variable which can be split into several temperature intervals. So to isolate the equivalent solutions, we ``fix'' a temperature interval and 
subsequently solve the MILP transshipment model. 

\noindent
\begin{minipage}{\linewidth}
\centering
\captionof{table}[]{HENS transshipment model symbols \citep{papoulias:1983}. Regular expressions denote alternatives, e.g.\ expression  $F(C_p)_{[i,j]}$ represents  $F(C_p)_{i}$ and  $F(C_p)_{j}$, the capacity of hot stream $i$ and cold stream $j$, respectively.}\label{table:hen}
\begin{tabular}{l l l}\toprule
 Name & Units & Description\\
\midrule
  {\bf Sets} & & \\
        \multicolumn{2}{l}{$HP, \, CP$} & Hot/Cold process streams\\
        \multicolumn{2}{l}{$HU, \, CU$} & Hot/Cold utilities\\
        \multicolumn{2}{l}{$HS, \, CS$} & Hot/Cold streams \& utilities\\
        \multicolumn{2}{l}{$TI = \{1,\dots, T\}$} & Temperature intervals\\
\midrule
{\bf Indices} \\
\multicolumn{2}{l}{$i \in HS=\{HP  \cup HU\}$} & Hot process stream/utility\\
\multicolumn{2}{l}{$j \in CS=\{CP  \cup CU\}$} & Cold process stream/utility\\
\multicolumn{2}{l}{$t \in TI$} & Temperature interval\\
\midrule
{\bf Parameters} & \\
 $F(C_p)_{[i,j]}$ & $[kW / K]$ & Flow rate capacities\\
$T^{[in,out]}_{[i, j,CU,HU]}$ & $[ K ]$ & Inlet/Outlet temperatures\\
$\delta T_{[i,j] t}$ & $[K]$  & Temperature changes at $t$\\
$Q{S_{it}}$ & $[kW]$ & Heat load of HU entering $t$\\
$Q{W_{jt}}$ & $[kW]$ & Heat load of CU exiting $t$\\
$Q_{[i,j]t}$ & $[kW]$ & Heat loads at $t$\\
$R_{t}$ & $[kW / K]$ & Total heat residual exiting $t$\\
$\Delta R_t$ & $[kW]$ & Heat residual difference at $t$\\
%
\midrule
{\bf Variables} & \\
$y_{ijt}$ & $[0, \, 1]$ & Existence of match $(ij)$ at $t$\\
$q_{ijt}$ & $[kW]$ & Heat load of $(ij)$ at $t$\\
$R_{it}$ & $[kW / K]$ & Heat residual of HS exiting $t$\\
$U_{ij}$ & $[kW]$ & Upper bound of match $(ij)$\\
\bottomrule
\end{tabular}
\end{minipage}
\vspace{-0.1cm}
\subsection*{MILP transshipment models}
\vspace{-0.1cm}
We assume the temperature change $\delta T_t$ is constant and equivalent for all streams in a fixed temperature interval.
The heat loads $Q_{it}$ provided and $Q_{jt}$ required by the relevant subset of hot streams $HP$ and cold streams $CP$, respectively, are: 
\begin{equation}\label{eq:henload} 
Q_{[i,j]t} = \delta T_t F(C_p)_{[i,j]}
\end{equation}
In the MILP transshipment model, the hottest hot utility is in the top temperature interval and a cold utility in the bottom interval; utilities are treated as streams in intermediate intervals. Excess heat is transferred to the next interval via a heat residual. 
Figure \ref{figure:hen} represents a transshipment model interval \citep{papoulias:1983, floudas}. The overall energy balance in Figure \ref{figure:hen} is given by:
\begin{equation}\label{eq:hen}
R_{t} - R_{t-1} + \sum_{j \in CU} Q{W_{jt}} - \sum_{i \in HU} Q{S_{it}}  =  \sum_{i \in HP} Q_{it} - \sum_{j \in CP} Q_{jt}
\end{equation}
The LP transshipment model initially provides the utility duties of the system, $Q{S_{i1}}, Q{W_{jT}}, R_t, R_{t-1}$.
\noindent
\begin{figure}
\centering
\includegraphics[width=.95 \linewidth]{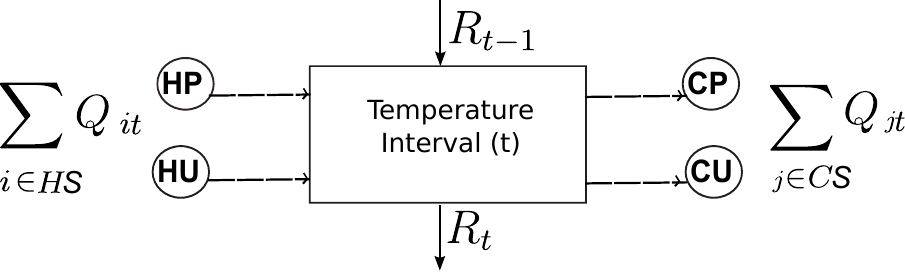}
\captionof{figure}[]{Heat balance around a temperature interval.}\label{figure:hen}
\end{figure}
The MILP transshipment model minimises the number of stream matches subject to thermodynamic constraints and prohibiting matches between hot utilities $HU$ and cold utilities $CU$. The objective function is a sum of binary variables where each term has coefficient 1. One way symmetry arises is due to the equal weights on each exchanger. \cite{chen:2015} suggest prioritising matches in potential heat exchange areas by reweighing the objective coefficients in the formulation of \cite{papoulias:1983}, i.e. the formulation used in this manuscript. But \cite{chen:2015} do not find significant computational advantages in reweighing the coefficients, so we chose to stick with the simpler formulation of \cite{papoulias:1983}.
The following formulation, illustrated in Figure \ref{figure:hen}, represents one fixed temperature interval $t=t'$:
\begin{align*}
\begin{array}{lll}
{\min} & \multicolumn{2}{l}{\displaystyle\sum_{i \in HP} \displaystyle\sum_{j \in CP} y_{ijt'}} \\
{\parbox{\widthof{$\min$}}{\textrm{s.t.}}} 
& Q_{it'}    = R_{it'} - R_{it'-1} + \displaystyle\sum_{j \in CS} q_{ijt'}, & i \in HP \\ 
& Q_{it'}    = R_{it'} - R_{it'-1} + \displaystyle\sum_{j \in CP} q_{ijt'}, & i \in HU \\ 
& Q_{jt'}  = \displaystyle\sum_{i \in HS} q_{ijt'}, & j \in CP \\ 
& Q_{jt'}  = \displaystyle\sum_{i \in HP} q_{ijt'}, & j \in CU \\ 
& R_{t'}   = \displaystyle\sum_{i \in HS} R_{it'},\\
& q_{ijt'} = 0, & i \in HU, j\in CU\\ 
& q_{ijt'} \leq \text{min}\{Q_{it'} + R_{it'-1}, Q_{jt'}\} y_{ijt'}, & i\in HS, \, j\in CS\\
& R_{it'} \geq 0, \, q_{ijt'} \geq 0 & i \in HS, \, j \in CS \\
& y_{ijt'}   \in \{0,\,1\} & i\in HS, \, j \in CS
\end{array}
\end{align*}
%
%
%
%
%
%
The MILP formulation exhibits combinatorial explosion in the possible configurations of hot and cold stream pairs. For example, in Figure \ref{figure:Hmatch}, $C_1$ can receive heat from $3$ hot streams. Similarly, $C_2$ and $C_3$, no matter how many $C_1$ matches, also require heat from one of the $3$ choices. Hence there are $3^3 = 27$ such configurations. Since $Q_{jt'}\neq 0$ $\implies$ $\displaystyle\sum_{i \in HS} q_{ijt'} > 0$, $j \in CP$, i.e.\ each cold stream needs to match with at least one hot stream in order to satisfy the load requirements. In the worst case scenario of two sets of $n$ hot streams and $m$ cold streams, their match is restricted either as one to one or one to many. Hence the following Lemma holds:
\begin{lemma}\label{lem:comb}
There are $n^m$ such configurations.
\end{lemma}
\begin{proof}[
Proof] of Lemma \ref{lem:comb} is provided in Appendix B\ref{appendix:B}.
\end{proof}
\section*{Mathematical Analysis}

This section offers background material on group theory and symmetry that help study HENS. Using the MILP transshipment formulation in the previous section, this section investigates network topology and detects symmetry in HENS.

\subsection*{Background on Symmetry}

The following definitions and descriptions are from \cite{clark:1984}.
A group $(\G, \cdot)$ is a nonempty set $\G$ with a binary operation $\cdot$ on $\G$ satisfying the following properties:
\begin{itemize}
\item[$\cdot$]  If $w, \, z$ $\in \G$, then $w\cdot z$ is also in $\G$;
\item[$\cdot$]  $w\cdot (z\cdot d)=(w\cdot z)\cdot d$ for all $w, \, z, \, d \in \G$;
\item[$\cdot$]  $\exists$ $\e \in \G$ such that $w\cdot \e = \e\cdot w = w, \; \forall w \in \G$;
\item[$\cdot$]  If  $w \in \G, \, \exists$ $w^{-1} \in \G$ such that $w\cdot w^{-1} = w^{-1}\cdot w = \e$.
\end{itemize}
A permutation of a set $X$ is a bijective function ${\bf \pi: \X\longrightarrow \X}$. Let $\X = \{1, \, \ldots, \, n\}$ be a set of $n$ elements and $\Pi^n$ the set of all permutations of elements in $\X$. To define the composition of groups, \cite{clark:1984} uses the notion of external and internal direct product. Given a finite sequence of groups $K_1, \, \ldots, \, K_n$, their external direct product is ${\Pi}_{i=1}^n K_i = K_1 \times K_2 \times \dots \times K_n$, with elements the tubles $(k_1, \, k_2, \, \ldots, \, k_n)$ for each $k_i \in K_i, \, \forall i$ and the operation of their product $(k_1, \, \ldots, \, k_n)(k_{1'}, \, \ldots, \, k_{n'}) = (k_1k_{1'}, \, \ldots, \, k_nk_{n'})$. Given a finite sequence of subgroups $K_1, \, \ldots, \, K_n \leq \G, \, \G$ is their internal direct product if the following properties hold:
\begin{itemize}
\item[$\cdot$] $\G = K_1 \dots K_n$ with tuples $(k_1\dots k_n)$ for $k_i \in K_i$;
\item[$\cdot$] for each $K_i \cap K_j$ a trivial subgroup is generated for $i \neq j$;
\item[$\cdot$] each $K_i$ is a normal subgroup of $\G$ i.e.\ $\forall k \in K_i$, $\forall g \in \G$, $gkg^{-1} \subseteq K_i$ and is denoted as $K_i \lhd G$.
\end{itemize}
The set $S_n: \X \longrightarrow \X$, under the operation of composition between all the $n$ distinct elements of $\X$, is the symmetric group of order $n$. The symmetry group of a MILP optimisation problem is a set of permutations that map any feasible solution to another feasible solution with the same value.

\subsection*{Detect Symmetry in HEN}

 Applying the formulation and definitions on several cases and combinations of streams, lead to the statement that in a fixed interval in HEN there can appear local symmetries through the exchange of streams. These symmetries arise in the feasible set of solutions from the topology of the problem as several pairs based on the binary variable $y_{ijt}$ give the same number of matches. This yields to the observation that if solutions in a temperature interval have the same number of matches and change of residual $\Delta R_{t'} = R_{t'} - R_{t'-1}$  then they are related between them. The symmetry group of a HEN is defined as follows: For $i \in HP$, $j \in CP$ and $t = t'$ if $\displaystyle\sum_{i \in HP} \displaystyle\sum_{j \in CP} y_{ijt'} =\displaystyle\sum_{i \in HP} \displaystyle\sum_{j \in CP} {y'}_{ijt'}$ and $\,  \Delta R_{t'} = \Delta {R'}_{t'} $ then $\exists \sigma \in \Pi^{n} \, \text{and} \, \pi \in \Pi^{m}$ such that:
\begin{equation}\label{def:symHEN}
\mathcal{G}(\textsl{HEN}(t'),  \Delta R_{t'}) \cong \{\sigma: HP \longrightarrow HP, \pi: CP \longrightarrow CP \}. 
\end{equation}
\section*{Mathematical Results}
\vspace*{-0.1cm}
\cite{liberti:2012} and \cite{costa-etal:2013} were the first to use algebra groups to explain the structure of optimisation problems and represent the symmetry. The properties of these groups then generate symmetry breaking constraints which improve the time that is taken for the problems to be solved. In this section, descriptions and proofs specify under which cases streams and utilities are considered to be symmetric in a temperature interval. 
 \subsection*{Represent Symmetry in HEN}
 For each stream the heat flow $F(C_p)_{[i,j]}$ is considered and Eq. \eqref{eq:hen} is rewritten.
\begin{equation}\label{eq:pload}
\Delta R_{t'} = \delta T_{t'} \Big (\displaystyle\sum_{i \in HP} F(C_p)_i  - \displaystyle\sum_{j \in CP} F(C_p)_j\Big).
\end{equation}
If two hot or two cold streams have the same flow rate heat capacity then they are equivalent.
\begin{lemma}\label{lem:HS}
For hot streams $h_1, h_2 \in HP$ if $F(C_p)_{h_1} = F(C_p)_{h_2}$ then $\exists$ a permutation $\sigma \in \Pi^n$ such that $\sigma(h_1) = h_2$.
\end{lemma}
\begin{proof}
Let $F(C_p)_{h} = F(C_p)_{h_1} = F(C_p)_{h_2}$, and since temperature interval is assumed to be constant then from Eq. \eqref{eq:pload}:
\begin{equation*}
\begin{aligned}
\Delta R_{t'} &= \delta T_{t'} \Big( F(C_p)_{h_1} + F(C_p)_{h_2}  - \displaystyle\sum_{j \in CP} F(C_p)_j \Big)\\ &= \delta T_{t'} \Big(F(C_p)_{h} + F(C_p)_{h} - \displaystyle\sum_{j \in CP} F(C_p)_j \Big).
\end{aligned}
\end{equation*}
\end{proof}

\begin{lemma}\label{lem:CS}
For cold streams $c_1, c_2 \in CP$ if $F(C_p)_{c_1} = F(C_p)_{c_2}$ then $\exists$ a permutation $\pi \in \Pi^m$ such that $\pi(c_1) = (c2)$.
\end{lemma}
\begin{proof}
Same as Lemma \ref{lem:HS} for cold stream.
\end{proof}
Since the change of temperature $\delta T_{t'}$ is the same and constant it can be trivially claimed that the above results hold for the cases where hot streams and utilities or cold streams and utilities have the same heat load $Q_{[i,j]t'}$ and can be exchanged between them in the possible matches. For the cases where there exist more than one type of streams with the same heat capacities the idea of \citet{ostrowski-etal:2015} for the unit commitment problem is contemplated in this part. He distinguishes the units with the same characteristic into classes and proves that the structure of the problem can enforce the branching strategies that he proposes when the problem is solved with the branch and bound algorithm.  

Let $\textbf{W}$ be the set of classes of equivalent hot streams and $n_{w}$ the number of streams in each class. 
\begin{prop} \label{prop:HS}
For $t = t'$, given a set ${HP}^w = \{h_1, \dots, h_{n_w}\}\subset HP$ with $F(C_p)_i = F(C_p)_{i^{'}} $ the symmetry group of ${HP}^w$: 
\begin{equation*}
 \mathcal{G}({HP}^w) \cong \{\sigma \in \Pi^n | \sigma:  {HP}^w \longrightarrow {HP}^w\} \cong S_{n_w}.
 \end{equation*}
\end{prop}

\begin{proof}[
Proof] Proved in Appendix B \ref{appendix:B}.
\end{proof}

Let $\textbf{Z}$ be the set of classes of equivalent cold streams and $n_{z}$ the number of streams in each class.  
\begin{prop} \label{prop:CS}
For $t = t'$, given a set ${CP}^z = \{c_1, \dots, c_{m_z}\}\subset CP$ with $F(C_p)_j = F(C_p)_{j'} $  the symmetry group of ${CP}^w$:
\begin{equation*} \mathcal{G}({CP}^z) \cong \{\pi \in \Pi^m | \pi:  {CP}^z \longrightarrow {CP}^z\} \cong S_{m_z}.  \end{equation*}
\end{prop}
\begin{proof}
Follows Proposition \ref{prop:HS} for cold streams.
\end{proof}

\begin{prop} \label{prop:HS2}
Let $\mathcal{G}({HP}^1)$, $\mathcal{G}({HP}^2)$,$\dots$, $\mathcal{G}({HP}^w)$ be the sequence of finite groups $S_{n_1}, S_{n_2}, \dots, S_{n_w}$.
 \begin{equation*} \mathcal{G}(HP) \cong  \mathcal{G}({HP}^1) \times \mathcal{G}({HP}^2) \times \dots \times \mathcal{G}({HP}^w) \cong S_{n_1} \times \dots \times S_{n_w}. \end{equation*}
\end{prop}

\begin{proof}
Follows a relevant Proof that is provided in \citep{liberti:2012, costa-etal:2013}.
\end{proof}

\begin{prop} \label{prop:CS2}
Let $\mathcal{G}({CP}^1)$, $\mathcal{G}({CP}^2)$,$\dots$, $\mathcal{G}({CP}^z)$ be the sequence of finite groups $S_{m_1}, S_{m_2}, \dots, S_{m_z}$: \begin{equation*} \mathcal{G}(CP) \cong \mathcal{G}({CP}^1) \times \mathcal{G}({CP}^2) \times \dots \times \mathcal{G}({CP}^z) \cong  S_{m_1} \times \dots \times S_{m_z}. \end{equation*}
\end{prop}
\begin{proof}
Same as proof of Proposition \ref{prop:HS2} for cold streams.
\end{proof}

\begin{theorem}
For $t = t'$ with sets of classes of equivalent hot and cold streams ${HP}^w\subset HP$, ${CP}^z\subset CP$ the symmetry group that describes the relations of all streams in the interval is given by: $\mathcal{G}(\textsl{HEN}(t'), \Delta R_{t'}) \cong \mathcal{G}(HP) \times \mathcal{G}(CP)$.
\end{theorem}
\begin{proof}
Follows from Proposition {\ref{prop:HS2}, \ref{prop:CS2}} and the definition of the internal and external direct product of groups. 
\end{proof}
These proofs can also be trivially generalised for the cases where hot streams and utilities and cold streams and utilities have the same heat load from the Eq. \eqref{eq:henload} and Lemma \ref{lem:HS}, Lemma \ref{lem:CS} which consist the bottleneck of the above results. 
\vspace*{-0.1cm}
\section*{Computational Tests}
The proofs of symmetry in HEN and the above observations demonstrated in a test case\footnote[1]{http://minlp.org/library/problem/index.php?i=191$\&$lib=MINLP} of a transshipment model as formulated and implemented on GAMS 24.7.1 by \cite{chen:2015}. The model has been tested on a single 3.40 GHz Intel(R) Core(TM) i7-4770 CPU of a computer with 130 GB memory and running Linux. The MILP solver CPLEX 12.6 is used with optimality gap set to be $10^{-3}$ and absolute gap $0.99$. The flow capacities of the streams that are used lie in a range that are closed to each other and considered as balanced streams. Authors report that tests with balanced streams show exponential increase in the number of nodes and the termination time, which is expected as there are much more combinations of pairs that can take place. From a CPLEX's feature ''solnpool'' several optimal solutions with the same objective value are obtained. The given data of the streams and how they take part at each of the three subnetworks in which the problem is solved are analysed.
\subsection*{Test Case}
\vspace{-0.1cm}
The instance that is tested here is the $Transshipment\_V1\_5$. It consists of 5 hot streams and 5 cold streams and 2 hot utilities and 1 cold utility as initially provided and obtained from the LP transshipment model in the output analysis. The problems data are shown in Table \ref{table:test1} and ten solutions of the MILP model with objective value 24 are presented analytically in Figure \ref{figure:case1}.
\vspace{2pt}
\begin{minipage}{\linewidth}
\centering
\captionof{table}[]{Data for Test Case.}\label{table:test1}
\begin{tabular}{l l l}\toprule[1pt]
\small
Streams [i,j] & $F(C_p)_{[i,j]}$ [kW/K] & $Q_{[i,j]}$ [kW]\\
\hline
5 HP & [1;2;1.5;1.7;2.5] & [280;440;345;442;500]\\ 
5 CP & [1.3;1.5;1.9;2.5;2.8] & [195;360;570;625;504]\\
2 HU & - &[110;195]\\
1 CU & - &[60]\\ 
\bottomrule[1pt]
\end{tabular}
\end{minipage}\vspace{1pt}
\noindent
\begin{minipage}{\linewidth}
\centering
\includegraphics[width=1\linewidth]{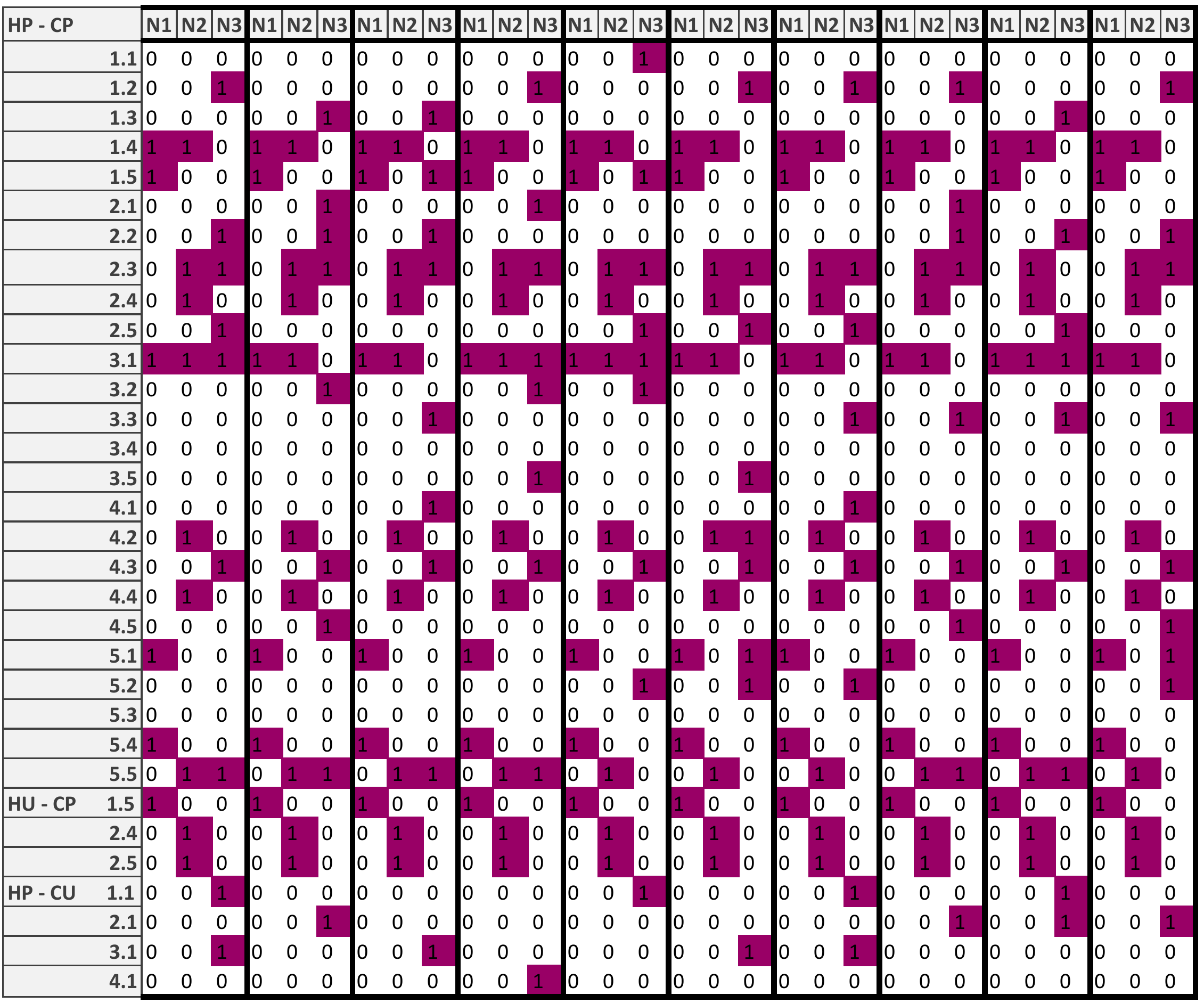}
\captionof{figure}[]{Matches of streams/utilities at each subnetwork.}\label{figure:case1}
\end{minipage}
\vspace{2pt}


The last subnetwork is isolated as the configuration of matches vary in each solution. It is observed that the load that is provided by $H1$ and $H3$ are the same and the load that is required by $C1$ and $CU$ are also the same. What is interesting is that even if initially they have different heat capacities in this subnetwork following the proofs from previous section these streams can be  exchanged. As illustrated in Figure \ref{figure:case12} these permutations lead to other optimal solutions some of which appear in the set of these solutions.
\subsection*{Observations}
This test case is relatively small, but note that it captures all the complexity investigated in this paper; we see significant symmetry even in this small example.
What is crucial in the results and observations is that if we are able to represent the relation of the parameters of a \textsl{HEN} problem from the given data then potentially the duplications of the identical solutions can be eliminated when the problem is solved. At the same time though is important that if at least one such optimal solution is produced all the others can be generated from the symmetry group that is assigned to the problem. Hence their effect in the overall investment cost can be evaluated which is sequentially solved in the last part of the formulation, that is not examined in this work.  
Moreover there is a prospective of using the results of this paper and reformulating the problem by introducing weight factors to prioritise or prohibit symmetric matches in the objective function.

\begin{minipage}{\linewidth}
\centering
\includegraphics[width=\linewidth]{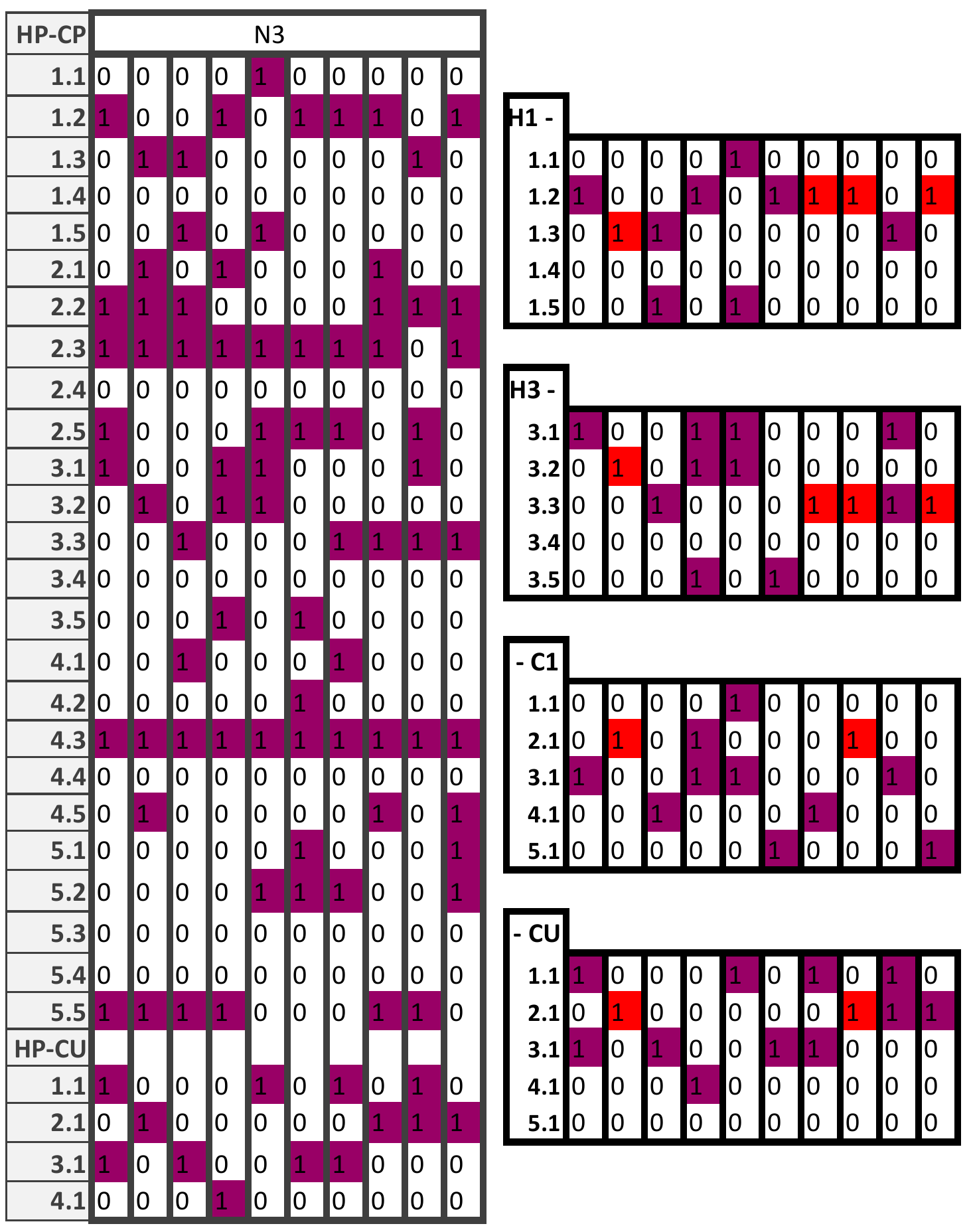}
\captionof{figure}[]{Matches of streams/utilities in subnetwork three.}\label{figure:case12}
\end{minipage}
\vspace{-0.6cm}
\section*{Conclusion}
This paper has explored, for the first time, the symmetric structure of the \textsl{HENS} problem. It has been shown where the symmetry is revealed in the problem and under which cases. We are interested in the performance of the algorithms that are used to solve these problems. More precisely the properties of the symmetry groups that we proved that they represent symmetry can be potentially used to exploit symmetry.
\section*{Acknowledgments}
\noindent
The support of the EPSRC DTP funding to G.K., and a Royal Academy of Engineering Research Fellowship to R.M., and EPSRC Grant EP/M028240/1, is acknowledged.
\section*{Appendices}
\subsection*{Appendix A}\label{appendix:A}
Original MILP transshipment model \citep{papoulias:1983}:
\begin{align*}
\begin{array}{lll}
\text{min} & \displaystyle\sum_{i \in HS} \displaystyle\sum_{t \in TI} \displaystyle\sum_{j \in CS} y_{ijt} \\ 
\text{s.t.} & R_{it} - R_{it-1} + \displaystyle\sum_{j \in CS} q_{ijt} = Q_{it},\, i\in HP, \, t \in TI \\
& R_{it} - R_{it-1}   + \displaystyle\sum_{j \in CS} q_{ijt} = Q_{it},\, i\in HU, \, t \in TI \\
& \displaystyle\sum_{i \in HS} q_{ijt}  = Q_{jt},\,  j\in CP, \, t \in TI \\ 
& \displaystyle\sum_{i \in HS } q_{ijt} = Q_{jt}, \, j\in CU,  t \in TI \\ 
& R_{t} -\displaystyle\sum_{i \in HS} R_{it} = 0, t \in TI \\ 
& Q_{ij} -  \displaystyle\sum_{t \in TI} q_{ijt} = 0,  i\in HS, \, j\in CS, \, t \in TI \\
& 0   \leq Q_{ij}  \leq U_{ij}y_{ijt},    i\in HS, \, j\in CS, \, t \in TI \\
& q_{ijt} \geq 0,   R_{it} \geq 0, i\in HS, \, j\in CS, \, t \in TI \\
& R_0 = R_t = 0, \, i\in HS, \, j\in CS, \, t \in TI \\
& y_{ij} \in \{0,\,1\}, \, i\in HS, \, j\in CS
\end{array}
\end{align*}
\vspace{-0.1cm}
\subsection*{Appendix B}\label{appendix:B}
\begin{proof}[Proof of Lemma \ref{lem:comb}]\label{proof:1}
Possible configurations of $n$ hot, $m$ cold streams\\ = $\underbrace{\binom{n}{1}\dots \binom{n}{1}}_{m} = \frac{n!}{(n-1)!}\dots \frac{n!}{(n-1)!}= \frac{n(n-1)!}{(n-1)!}\dots n = n^m $
\end{proof}
\vspace{-0.1cm}
\begin{proof}[Proof of Proposition \ref{prop:HS}]\label{proof:1}
Claim: there are $n!$ such permutations between all the elements of ${HP}^{w}$.\\ This can be shown as follows \citep{clark:1984}:
\begin{itemize}
\item [(1)] Assign $\sigma (h_1)$ to one of the elements of ${HP}^{w}$: there are n such choices
\item [] Since $\sigma$ is bijective $\sigma (h_1) \neq \sigma (h_2)$:
\item[(2)] Assign $\sigma (h_2)$ to one of the remaining elements of ${HP}^{w} - \{\sigma (h_1 )\}$: there are $(n-1)$ such choices
\item [$\dots$]
 \item [(n)] Assign $\sigma(h_{n_w})$ to the only remaining element: there is only $1$ such choice
\end{itemize}
Hence there are $n(n-1)\dots 1 = n!$ such permutations under which ${HP}^{w}$ is invariant.
\end{proof}
\linespread{1.2}
\bibliographystyle{apalike}
\small
\bibliography{symm_groups}

\end{document}